\documentclass[12pt]{article}
\usepackage{amssymb,amsmath, amsthm,latexsym, verbatim, amscd}

\def \Hsum#1{{{\underset{#1}{{\sum}^{\oplus}}}}}
  \newtheorem{theorem}{Theorem}[section] %
  \newtheorem{proposition}[theorem]{Proposition} %
  \newtheorem{lemma}[theorem]{Lemma} %
  \newtheorem{corollary}[theorem]{Corollary} %
  
  \newtheorem{fact}[theorem]{Fact}
\theoremstyle{definition} %
  \newtheorem{definition}[theorem]{Definition} %
  \newtheorem{example}[theorem]{Example} %
  
\theoremstyle{remark} %
  \newtheorem{remark}[theorem]{Remark} %
\begin{document}
\title{Admissible restrictions of irreducible representations
        of reductive Lie groups: 
symplectic geometry and discrete decomposability
\\
\normalsize{---
Dedicated to Bertram Kostant with admiration to his deep and vast perspectives
 and with sincere gratitude to his constant encouragement for many years}} %
\author{Toshiyuki KOBAYASHI
\\
Graduate School of Mathematical Sciences
 and Kavli IPMU (WPI)
\\
The University of Tokyo}

\date{} %

\maketitle %

\begin{abstract}
Let $G$ be a real reductive Lie group,
 $L$ a compact subgroup, 
 and $\pi$ an irreducible admissible representation of $G$.
In this article we prove a necessary and sufficient condition 
 for the finiteness of the multiplicities of $L$-types
 occurring in $\pi$
 based on symplectic techniques.  
This leads us to a simple proof 
 of the criterion
 for discrete decomposability of
 the restriction of unitary representations
 with respect to noncompact subgroups (the author, Ann.~Math. 1998), 
 and also provides a proof of a reverse statement 
 which was announced in [Proc.~ICM 2002, Thm.~D].  
A number of examples are presented 
 in connection with Kostant's convexity theorem
 and also with non-Riemannian locally symmetric spaces.  
\end{abstract}

\newcommand{\gk}{$(\frak g, K)$}
\newcommand\domch{{\frak t_+^*}}
\newcommand\Hom{\operatorname{Hom}}
\newcommand\adm[1]{{${#1}$-admissible}}
\newcommand\Kasym[1]{\operatorname{AS}_K({#1})}
\newcommand\Ad{\operatorname{Ad}}
\newcommand{\SSK}[1]{\operatorname{Supp}_K({#1})} 
\newcommand\mSSK[1]{\operatorname{Supp}_K({#1})} 
\newcommand\posspan[1]{{\mathbb {#1}}_{\ge 0}\text{-{\rm{span}}}\,}
\newcommand\pr[2]{\operatorname{pr}_{{#1}\to{#2}}}
\newcommand\wtlat{\Lambda_+}
\newcommand\Image{{\operatorname{Image}}}
\newcommand \set[2]{\{{#1}:{#2}\}}
\newcommand\reg[1]{{\Bbb C[{#1}]}} 

\newcommand \itm[1]{\newline\indent{\rm{#1}}\enspace}
\newcommand\Ind{{\operatorname{Ind}}}

\section{Introduction and Statement of Main Results}
This article is a continuation of \cite{xkdecomp, xkdecoalg, xkdecoass},
 where we studied
 the restriction of an irreducible unitary representation $\pi$ of
 a real reductive Lie group $G$ with respect to a reductive subgroup $G'$.
There,
 we highlight branching laws 
{\bf{without continuous spectrum}}.  
As we mention in Section \ref{subsec:1.8} below, 
 a key to discrete decomposability is {\it{$K'$-admissibility}} of $\pi$
 (\cite[Thm.~1.2]{xkdecomp}), 
 that is, 
\begin{equation}
\label{eqn:mfinite}
   \dim_{\mathbb{C}} \Hom_{K'}(\tau, \pi|_{K'}) < \infty
    \quad
    \text{for any}
    \quad
    \tau \in \widehat{K'}, 
\end{equation}
where $K'$ is a maximal compact subgroup of $G'$.

In this article we prove a necessary and sufficient condition
 for the $K'$-admissibility 
 of irreducible \gk-modules $X$ with $K' \subset K$.

\subsection{Two closed cones $\operatorname{A S}_K(X)$ and $C_K(K')$}
\label{subsec:3}
In order to state our main results,
 let us fix some notation.
 
Let $G$ be a connected linear reductive Lie group, 
 $K$ a maximal compact subgroup of $G$,
 and $T$ a maximal torus of $K$.
Their Lie algebras will be denoted 
 by the lowercase German letters.  
Fix a positive system $\Delta^+(\frak k_{\mathbb{C}}, \frak t_{\mathbb{C}})$,
  and we write $\frak t_+^*$ $(\subset \sqrt{-1} {\mathfrak{t}}^{\ast})$
 for the dominant Weyl chamber.
The set of dominant weights 
 which lift to the torus $T$ is denoted by $\Lambda_+$.  
It is a submonoid of ${\mathfrak{t}}_+^{\ast}$
 (that is, 
 it contains 0 and is invariant under addition).  
The Cartan--Weyl highest weight theory
 for the group version
 establishes a bijection between $\widehat{K}$ with $\wtlat$.
We shall denote by $V_\mu$ the irreducible representation of $K$
 with highest weight $\mu \in \wtlat$.

For a subset $S$ in a Euclidean space $E$,
 the limit cone $S\infty$ is the set of $E$
  consisting of all elements of the form $\lim_{j\to \infty} \varepsilon_j \mu_j$
  for some sequence  $(\mu_j, \varepsilon_j) \in S \times \Bbb R_+$
 with $\lim_{j\to\infty} \varepsilon_j=0$
 (\cite[Def.~2.4.2]{xkkk}).
The {\it{asymptotic $K$-support}} $\Kasym{X}$ of a $K$-module $X$ is defined 
 to be the limit cone
 of the $K$-support of $X$
 (Kashiwara--Vergne \cite{xkvktype}):
\begin{alignat}{2}
  \SSK{X} &:=\set{\mu \in \wtlat}{\Hom_K(V_\mu, X) \neq \{0\}}\,\,
  &&\subset \Lambda_+,
\label{eqn:SuppK}
\\
  \Kasym{X} &:= \SSK{X} \infty &&\subset{\mathfrak{t}}_+^{\ast}.
\label{eqn:ASK}
\end{alignat}

Let $K'$ be a closed subgroup of $K$, 
 and set $(\frak k')^\perp:=\{\lambda \in {\mathfrak{k}}^{\ast}:\lambda|_{\mathfrak{k}'}\equiv 0\}$.
We regard ${\mathfrak{t}}^{\ast}$ as a subspace 
 of ${\mathfrak{k}}^{\ast}$
 via a $K$-invariant inner product on ${\mathfrak{k}}$, 
 and define a closed cone in $\sqrt{-1}{\mathfrak{t}}^{\ast}$
 by 
\begin{equation}
\label{eqn:1.5.1}
   C_K(K') := \frak t_+^* \cap \sqrt{-1} \Ad^*(K)(\frak k')^\perp.
\end{equation}

These two closed cones $\operatorname{A S}_K(X)$ 
 and $C_K(K')$ are a finite union 
 of convex polyhedral cones
 (Propositions \ref{prop:2.5} and \ref{prop:2.2}, respectively).

\subsection{Criterion for finite multiplicities}
\label{subsec:1.6}
Here is our main theorem:
\begin{theorem}
\label{thm:A}
Let $X$ be a \gk-module of finite length, 
 and $K'$ a closed subgroup of $K$.  
Then the following two conditions are equivalent:
\itm{(i)} $X$ is \adm{K'};
\itm{(ii)}
 $\Kasym{X} \cap C_K(K') = \{0\}$.
\end{theorem}

Some remarks are in order.
\newline
(1) The main result of \cite{xkdecoalg} was a discovery
 of the criterion (ii) in Theorem \ref{thm:A}, 
 and the implication (ii) $\Rightarrow$ (i) was proved 
 in \cite[Thm.~2.8]{xkdecoalg}
 based on micro-local study:
 the asymptotic $K$-support $\operatorname{A S}_K(X)$
 played a role in an estimate of the singularity spectrum
 of the hyperfunction character of $X|_K$.
In this article 
 we give a new and simple proof for the implication (ii) $\Rightarrow$ (i)
 based on symplectic geometry:
 the cone $C_K(K')$ is interpreted as the momentum set
 for the natural Hamiltonian action
 on the cotangent bundle $T^{\ast}(K/K')$, 
 see Section \ref{subsec:2.2}.  
\newline
(2) In this article, 
 we also give a proof
 of the reverse implication (i) $\Rightarrow$ (ii).  
This statement was announced 
 in the proceeding of ICM 2002 \cite[Thm.~D]{xkICM2002}, 
 and a sketch of the proof was given
 in the lecture notes \cite[Chap.~6]{deco-euro}, 
 however, 
 the full proof has not been published until this article.
\newline
(3) \enspace
Theorem~\ref{thm:A} still holds for disconnected groups,
 namely, we may allow $K$
  to have finitely many connected components.  
In this case, 
 the same proof works by using the asymptotic $K_0$-support 
 of $X$ regarded as a $K_0$-module, 
 where $K_0$ is the identity component of $K$.

\subsection{Admissible restriction to noncompact subgroups}
\label{subsec:1.8}
Let $\pi$ be a unitary representation of $G$, 
 and $G'$ a subgroup.  
By the general theory
 of unitary representations of locally compact groups \cite{xmautner}, 
 the restriction $\pi|_{G'}$ is decomposed
 into the direct integral 
 of irreducible unitary representations of $G'$, 
 uniquely up to isomorphisms
 when $G'$ is reductive 
 \cite{xHCadm}, 
 as follows:
\begin{equation}
\label{eqn:branch}
   \pi|_{G'}
   \simeq
  \int_{\widehat{G'}}^{\oplus}
  m_{\pi}(\tau) d \mu(\tau)
  \quad
  \text{(direct integral)}, 
\end{equation}
where $\widehat{G'}$ denotes the unitary dual of $G'$, 
 that is, 
 the set of equivalence classes
 of irreducible unitary representations of $G'$, 
 $d \mu$ is a Borel measure of $\widehat{G'}$, 
 and $m_{\pi} \colon \widehat{G'} \to {\mathbb{N}} \cup \{\infty\}$
 is a measurable function.  
The irreducible decomposition \eqref{eqn:branch}
 is called the {\it{branching law}}
 of the restriction $\pi|_{G'}$, 
 and $m_{\pi}$ is the {\it{multiplicity}}.  
In general the branching law may involve continuous spectrum, 
 and the multiplicity $m_{\pi}$ may take infinite values.  
The following definition singles out
 a framework 
 in which we could expect a simple
 and detailed algebraic study
 of the restriction $\pi|_{G'}$
 ({\it{symmetry breaking}}, 
 {\it{cf.}} \cite{vogan60}).  

\begin{definition}
[{\cite[Sect.~1]{xkdecomp}}]
\label{def:admrest}
We say a unitary representation $\pi$ of $G$ is $G'$-{\it{admissible}}
 if $\pi$ splits into a direct sum 
 of irreducible unitary representations
 of $G'$
\[
   \pi|_{G'}
   \simeq
 \Hsum{\tau \in \widehat {G'}} m(\tau) \tau
\quad
\text{\rm{(Hilbert direct sum)}}
\]
 with multiplicity $m(\tau) < \infty$
 for all $\tau \in \widehat{G'}$.  
\end{definition}

If $G'$ itself is compact,
 then the decomposition \eqref{eqn:branch} is automatically discrete, 
 and thus,
 $G'$-admissibility is nothing but the finiteness
 of the multiplicity
 $m_{\pi}(\tau)$ for all $\tau$.  
In the general case where $G'$ is noncompact,
 we take a maximal compact subgroup $K'$ of $G'$.  
Then $K'$-admissibility implies 
 $G'$-admissibility 
 (\cite[Thm.~1.2]{xkdecomp}).  
Therefore, 
 as an immediate corollary of Theorem \ref{thm:A}, 
 we recover:
\begin{corollary}
[{\cite[Thm.~2.9]{xkdecoalg}}]
\label{cor:C}
Let $\pi  \in \widehat{G}$, and $G'$ a reductive subgroup of $G$.
If $\Kasym{\pi} \cap \sqrt{-1} \Ad^*(K)(\frak k')^\perp = \{0\}$,
 then the restriction $\pi|_{G'}$ splits into a {\bf discrete}
  sum of irreducible unitary representations of $G'$
 with {\bf{finite multiplicities}}.  
\end{corollary}

\subsection{Restriction of discrete series representations}
\label{subsec:DiscG}

It is plausible, 
 see \cite[Conj.~D]{aspm}, 
 that the converse of \cite[Thm.~1.2]{xkdecomp} also holds, 
 namely, 
 $G'$-admissibility is equivalent
 to $K'$-admissibility
 if the representation arises as the restriction of an irreducible unitary representation
 of a real reductive linear Lie group $G$
 to its reductive subgroup $G'$ 
 with maximal compact subgroup $K'$.  
If this conjecture is affirmative, 
 then the criterion in Theorem \ref{thm:A} will give 
 a necessary and sufficient condition
 for the restriction $\pi|_{G'}$ 
 to be $G'$-admissible.  
In this section
 we discuss such an example.

An irreducible unitary representation $\pi$ of $G$
 is called a {\it{square-integrable representation}}
 if it is realized in a closed invariant subspace
 of the regular representation on the Hilbert space $L^2(G)$.  
The isomorphism classes of all such irreducible, 
 square integrable representations
 constitute a subset
 $\operatorname{Disc}(G) \subset \widehat {G}$, 
 the {\it{discrete series}} of $G$.  
In this case, 
 the conjecture is true, 
 see \cite{xdgv, hcr-rest, xzhuliang}.  
By Theorem \ref{thm:A}, 
 we can detect whether $\pi$ is $G'$-admissible
 or not when restricted to a reductive subgroup $G'$:
\begin{corollary}
\label{cor:HCdeco}
Let $\pi$ be a square-integrable representation of $G$, 
 and $G'$ a closed reductive subgroup of $G$.  
Then the following four conditions on the triple $(G, G', \pi)$
 are equivalent:
\itm{(i)}
The restriction $\pi|_{G'}$ is $G'$-admissible.  
\itm{(i)$'$}
There is a map $m \colon \operatorname{Disc}(G') \to {\mathbb{N}}$
 such that
\[
  \pi|_{G'} \simeq \Hsum{\tau \in \operatorname{Disc}(G')} m(\tau) \tau
\qquad
\text{\rm{(Hilbert direct sum)}}.  
\]
\itm{(ii)}
The restriction $\pi|_{K'}$ is $K'$-admissible.  
\itm{(iii)}
$\Kasym{\pi} \cap \sqrt{-1} \Ad^{\ast}(K)({\mathfrak{k}}')^{\perp}
=\{0\}.  
$
\end{corollary}

\begin{remark}
{\rm{
In the case where $(G,G')$ is an irreducible symmetric pair,
 the triple $(G,G',\pi)$
 satisfying the criterion (iii) was classified
 in Kobayashi--Oshima \cite{decoAq}.  
We refer to \cite{xbrionc, Ko93, xkdecomp, xkhowe70, xvargas16}
 for some explicit formulas of discrete branching laws.  
On the other hand, 
 Duflo--Galina--Vargas \cite{xdgv} studied in detail
 the case where the subgroup $G'$ is isomorphic to $S L(2,{\mathbb{R}})$ 
 or $P S L (2,{\mathbb{R}})$.  
}}
\end{remark}

The proof of Theorem \ref{thm:A} and Corollary \ref{cor:HCdeco}
 is given in Section \ref{sec:2}.  
Applications of Theorem \ref{thm:A} are given 
 in connection with Kostant's convexity theorem
 for momentum maps 
 and with the boundaries of semisimple symmetric spaces
 in Sections \ref{sec:convex} and \ref{sec:4}, 
 respectively.  

\vskip 1pc
\par\noindent
{\bf{Notation:}}\enspace
${\mathbb{R}}_{\ge 0} :=\{x \in {\mathbb{R}}: x \ge 0\}$, 
${\mathbb{Q}}_{\ge 0} := {\mathbb{Q}} \cap {\mathbb{R}}_{\ge 0}$ 
 and ${\mathbb{N}}_{\ge 0}:= {\mathbb{N}} \cap {\mathbb{R}}_{\ge 0}$. 
 
\section{Proof of Main Results}
\label{sec:2}
In this section,
 we give an interpretation of the two invariants
 $\operatorname{A S}_{K}(\pi)$ and $C_K(K')$ from a viewpoint
 of symplectic geometry,
 and prove Theorem \ref{thm:A}. 

\subsection{Rational convex polyhedral cones}
\label{subsec:2.8}
Let $E$ be a finite-dimensional vector space over ${\mathbb{Q}}$, 
 and $S$ a finite subset of $E$.  
The {\it{convex polyhedral cone}} spanned by $S$ is the smallest convex cone
 in $E$, 
 that is, 
\[
  {\mathbb{Q}}_{\ge 0}\operatorname{-span} S
  =
  \{
   \sum_{j=1}^k a_j s_j
   :
   a_1, \cdots, a_k \in {\mathbb{Q}}_{\ge 0}, 
   s_1, \cdots, s_k \in S
\}.  
\]
Similarly,
 we can define ${\mathbb{Z}}_{\ge 0}\operatorname{-span} S$
 $(\subset E)$
 and ${\mathbb{R}}_{\ge 0}\operatorname{-span} S$
 $(\subset E \otimes_{\mathbb{Q}} {\mathbb{R}})$.

Here is an elementary observation of the intersections
 of two such polyhedral cones.  

\begin{lemma}
\label{lem:2.8}
Let $S, T$ be finite subsets of ${\mathbb Q}^n$.
Then the following four conditions on $S$ and $T$ are equivalent:
\itm{(i)} $\posspan{Z} S \cap \posspan{Z} T \neq \{0\}$;
\itm{(ii)} $\posspan{Q} S \cap \posspan{Q} T \neq \{0\}$;
\itm{(iii)} $\posspan{R} S \cap \posspan{R} T \neq \{0\}$;
\itm{(iv)}$(\delta\text{\rm{-neighbourhood of }}\posspan{R} S) \cap \posspan{R} T$
 is unbounded for some $\delta >0$.  
\end{lemma}

\begin{proof}
The implications 
 (i) $\Leftrightarrow$ (ii) $\Rightarrow$ (iii) $\Rightarrow$ (iv) are obvious.
The implication (iv) $\Rightarrow$ (iii) is immediate
 by taking the limit cone.  
For the remaining implication
 (iii) $\Rightarrow$ (ii), 
 we observe that the condition (iii) holds if and only if 
 $\posspan{R} S \cap \posspan{R} T$ contains a face of positive dimension,  say $W'$.
We extend $W'$ to the equi-dimensional subspace
 $W$ in $\mathbb R^n$.
Then $W$ is defined over $\mathbb Q$, 
 hence 
$
\posspan{Q} S \cap \posspan{Q} T \supset W' \cap \mathbb Q^n 
 \neq \{0\}.
$
Thus we have proved (iii) $\Rightarrow$ (ii).
\end{proof}

\subsection{Regular functions on affine $K_{\mathbb{C}}$-varieties}
\label{subsec:2.1}

Let $\mathcal V$ be an irreducible affine $K_\Bbb C$-variety over $\Bbb C$.
Then the ring $\reg{\mathcal V}$ of regular functions is finitely generated.
We need some basic fact on the $K_\Bbb C$-module structure
 of $\Bbb C[\mathcal V]$.  

\begin{lemma}
\label{lem:2.1}
The $K$-support $\SSK{\reg{\mathcal V}}$ is
 a finitely generated submonoid of $\Lambda_+$, 
 that is, there exist finitely many $\lambda_1, \dots, \lambda_k \in \wtlat$
 such that
$$
   \SSK{\reg{\mathcal V}} = \posspan{Z}\{\lambda_1, \dots, \lambda_k\}.
$$
\end{lemma}

For the convenience of the reader, 
 we review quickly its proof, 
 see \cite{xbrionc, xsjadv}.

\begin{proof}
We write $N(K_\Bbb C)$ for the maximal unipotent subgroup of $K_\Bbb C$
 corresponding to the positive system $\Delta^+(\frak k_{\mathbb{C}}, \frak t_{\mathbb{C}})$.  
Then the ring 
$
\reg{K_\Bbb C/N(K_\Bbb C)} \simeq \bigoplus_{\lambda \in \wtlat} V_\lambda
$
 is finitely generated since $V_\lambda V_\mu = V_{\lambda + \mu}$.
Then the left-hand side of the isomorphism:
$$
   \left(\reg{K_{\Bbb C}/N(K_\Bbb C)}\otimes\reg{\mathcal V}\right)^{K_{\Bbb C}}
   \simeq
   \reg{\mathcal V}^{N(K_\Bbb C)}
$$
 is finitely generated
 because $K_{\mathbb{C}}$ is reductive.  
Thus the ring $\reg{\mathcal V}^{N(K_\Bbb C)}$ is finitely generated,
 whence the $K$-support $\SSK{\reg{\mathcal V}}$ is finitely generated
 as a monoid.  
\end{proof}

\subsection{Hamiltonian actions and cotangent bundles}
\label{subsec:2.2}

Let $(M,\omega)$ be a symplectic manifold, 
 and $K$ a Lie group acting on $M$
 as symplectic diffeomorphisms.  
The action is called {\it{Hamiltonian}}
 if there exists a {\it{momentum map}}
 $\Phi \colon M \to {\mathfrak{k}}^{\ast}$
 with the property that
$
  d \Phi^Z = \iota(Z_M) \omega
$
for all $Z \in {\mathfrak{k}}$, 
where $Z_M$ denotes the vector field on $M$
 induced by $Z$, 
 and $\Phi^Z$ is the function on $M$
 defined by $\Phi^Z(m)=\Phi(m)(Z)$.  
The {\it{momentum set}} $\Delta(M)$ is defined by
\begin{equation}
\label{eqn:moment}
 \Delta(M):=\sqrt{-1} \Phi(M) \cap {\mathfrak{t}}_+^{\ast}.  
\end{equation}

Let $K'$ be a connected closed subgroup of $K$.  
The cotangent bundle $T^{\ast}(K/K')$ of the homogeneous space $K/K'$
 is given as a homogeneous vector bundle
 $K \times_{K'}({\mathfrak{k}}')^{\perp}$.  
Thus the symplectic manifold $T^{\ast}(K/K')$
 is a Hamiltonian $K$-space with moment map
\begin{equation}
\label{eqn:2.2.2}
   \Psi \colon T^*(K/K') \to \frak k^*,
   \quad
      (k, X) \mapsto \Ad^*(k) X.  
\end{equation}
Let $K_\Bbb C' \subset K_\Bbb C$ be the complexifications of $K' \subset K$.  
For the affine variety $K_{\mathbb{C}}/K_{\mathbb{C}}'$, 
 we take $\lambda_1, \dots, \lambda_k \in \wtlat$ 
 as in Lemma \ref{lem:2.1} such that
\begin{equation}
\label{eqn:2.2.1}
 \operatorname{Supp}_K(\reg{K_{\Bbb C}/{K_{\Bbb C}'}}) = \posspan{Z}\{\lambda_1, \dots, \lambda_k\}.  
\end{equation}

\begin{proposition}
\label{prop:2.2}
\begin{enumerate}
\item[{\rm{(1)}}]
The momentum set $\Delta(T^{\ast}(K/K'))$ is equal to $C_K(K')$.  
\item[{\rm{(2)}}]
$C_K(K')=\operatorname{A S}_K(C^{\infty}(K/K'))$.  
In particular,
 we have
\[
 C_K(K')
= \posspan{\Bbb R}\{\lambda_1, \dots, \lambda_k\}.  
\]
\end{enumerate}
\end{proposition}
\begin{proof}
(1)\enspace
It follows from the definitions \eqref{eqn:2.2.2} and \eqref{eqn:1.5.1} that
\begin{equation}
\label{eqn:2.2.3}
  \Delta(T^{\ast}(K/K'))
  =\sqrt{-1}\Ad^*(K) (\frak k')^\perp \cap \frak t_+^*
 = C_K(K').
\end{equation}
\itm{(2)}
By Sjamaar \cite[Thms.~4.9 and 7.6]{xsjadv}, 
 we have
\[
  \Delta(T^{\ast}(K/K'))
  =\Delta(K_{\mathbb{C}}/K_{\mathbb{C}}')
  = \posspan{\Bbb R}\{\lambda_1, \dots, \lambda_k\}.
\]
Combining this with \eqref{eqn:2.2.3}, 
 we get the second statement.  
\end{proof}

\subsection{Associated varieties}
\label{subsec:2.3}
The associated varieties ${\mathcal{V}}(X)$ are 
 coarse approximation of ${\mathfrak{g}}$-modules $X$, 
 which we brought in \cite{xkdecoass}
 into an algebraic study of discretely decomposable restrictions
 of Harish-Chandra modules.  
In this section we collect some important properties 
 of associated varieties, 
 and reduce the $K'$-admissibility 
 of a Harish-Chandra module on ${\mathcal{V}}(X)$
 to that of the space of regular functions on ${\mathcal{V}}(X)$.

Let $\{U_j({\mathfrak{g}}_{\mathbb{C}})\}_{j \in {\mathbb{N}}}$
 be the standard increasing filtration
 of the universal enveloping algebra $U({\mathfrak{g}}_{\mathbb{C}})$.  
Suppose $X$ is a finitely generated ${\mathfrak{g}}$-module.  
Let $F$ be a finite set of generators, 
 and we set $X_j:=U_j({\mathfrak{g}}_{\mathbb{C}})F$. 
The graded algebra $\operatorname{gr} U({\mathfrak{g}}_{\mathbb{C}})
:=\bigoplus_{j \in {\mathbb{N}}} U_j({\mathfrak{g}}_{\mathbb{C}})/U_{j-1}({\mathfrak{g}}_{\mathbb{C}})$ is isomorphic to the symmetric algebra $S({\mathfrak{g}}_{\mathbb{C}})$
 by the Poincar{\'e}--Birkhoff--Witt theorem
 and we regard the graded module
$
  \operatorname{gr} X:=
  \bigoplus_{j \in {\mathbb{N}}} X_j/X_{j-1}
$
 as a $S({\mathfrak{g}}_{\mathbb{C}})$-module.  
Define
\begin{align*}
 \operatorname{Ann}_{S({\mathfrak{g}}_{\mathbb{C}})}(\operatorname{gr} X)
 :=&
  \{f \in S({\mathfrak{g}}_{\mathbb{C}}):
    f v=0\quad \text{for any $v \in \operatorname{gr}X$}\}, 
\\
 {\mathcal{V}}(X)
  :=&
  \{x \in {\mathfrak{g}}_{\mathbb{C}}^{\ast}:
    f(x)=0\quad 
\text{for any $f \in \operatorname{Ann}_{S({\mathfrak{g}}_{\mathbb{C}})}(\operatorname{gr}X)$}
\}.  
\end{align*}
Then ${\mathcal{V}}(X)$ does not depend on the choice of $F$, 
 and is called the {\it{associated variety}} of $X$.  
If $X$ is a Harish-Chandra module, 
 that is, 
 a \gk-module of finite length, 
 then the associated variety ${\mathcal{V}}(X)$ is a $K_{\mathbb{C}}$-stable closed subvariety
 of ${\mathcal{N}}({\mathfrak{p}}_{\mathbb{C}}^{\ast})$, 
 see \cite{xvoass}.

For two $K$-modules $X_1, X_2$,
 we use the notation from \cite{xkdecomp}, and write $X_1 \le_K X_2$ if 
$$
  \dim_{\mathbb{C}} \Hom_{K}(\tau, X_1) \le \dim_{\mathbb{C}} \Hom_K(\tau, X_2)
  \quad
  \text{for any }
  \tau \in \widehat{K}.
$$
\begin{lemma}
[{\cite[Prop.~3.3]{xkyoshi15}}]
\label{lem:2.3}
Let $X$ be a \gk-module of finite length,
 and $\mathcal V(X)$ the associated variety.
We write 
$
\mathcal V(X) = \overline{\mathcal O_1} \,\cup \cdots \cup \,\overline{\mathcal O_N}
$
 for the decomposition into irreducible components.
Then there exist finite-dimensional representations $F_j$ $(1 \le j \le N)$
 of $K$ such that
\begin{align}
\label{eqn:2.3.2}
   X  &\le_K \bigoplus_{j=1}^N \reg{\overline{\mathcal O_j}} \otimes F_j, 
\\
   X \otimes F_j^*  &\ge_K   \reg{\overline{\mathcal O_j}}
 \ \text{ for any } j \quad(1 \le j \le N). 
\label{eqn:2.3.3}
\end{align}
\end{lemma}

\subsection{Basic properties of asymptotic $K$-support}
\label{subsec:2.4}
We recall some basic properties of asymptotic $K$-support defined in \eqref{eqn:ASK}.
\begin{lemma}
\label{lem:2.4}
Let $X$ and $Y$ be $K$-modules.
\itm{(1)}
If $Y \le_K X$ then $\Kasym{Y} \subset \Kasym{X}$.
\itm{(2)}
 $\Kasym{X} = \Kasym{X \otimes F}$
 for any finite-dimensional representation $F$ of $K$.
\itm{(3)}
$\Kasym{X \oplus Y} = \Kasym{X} \cup \Kasym{Y}$.
\end{lemma}
\begin{proof}
(1)\enspace Clear from $\SSK{Y} \subset \SSK{X}$.
\newline
(2)\enspace
See \cite[Lem.~3.1]{xkdecoalg}.  
\itm{(3)}
Immediate from $(S \cup T)\infty = S\infty \cup T\infty$
 for any subsets $S$ and $T$.  
\end{proof}

\subsection{Asymptotic $K$-supports of Harish-Chandra modules}
\label{subsec:2.5}
The asymptotic $K$-support
 $\Kasym{X}$ 
 of a Harish-Chandra module $X$
 is determined by its associated variety ${\mathcal{V}}(X)$, 
 and is a finite union of convex polyhedral cones.  
These properties will be used in the proof of Theorem \ref{thm:A}.

Suppose we are in the setting of Lemma \ref{lem:2.3}.
For each irreducible component $\overline{\mathcal O_j}$
 of the associated variety ${\mathcal{V}}(X)$,
 we take a finite set 
 $S_j:=\{\beta_1, \dots, \beta_{k_j}\}$ so that
$
 \SSK{\reg{\overline{\mathcal O_j}}} = \posspan{Z} S_j 
$
as in Lemma \ref{lem:2.1}.
Taking the limit cone, we have:
\begin{equation}
\label{eqn:2.5.2}
\Kasym{\reg{\overline{\mathcal O_j}}} = \posspan{R} S_j.
\end{equation}
\begin{proposition}
\label{prop:2.5}
Let $X$ be a \gk-module of finite length,
 and $S_j$ $(1 \le j \le N)$ finite subsets of $\wtlat$ as above.
 Then,
$\Kasym{X} = \Kasym{{\mathbb{C}}[{\mathcal{V}}(X)]}
           = \bigcup_{j=1}^N \posspan{\Bbb R}S_j$.
\end{proposition}
\begin{proof}
By Lemmas \ref{lem:2.3} and \ref{lem:2.4}, we have
$$
\Kasym{X} \subset \bigcup_{j=1}^N \Kasym{\reg{\overline{\mathcal O_j}}\otimes F_j}
         = \bigcup_{j=1}^N \Kasym{\reg{\overline{\mathcal O_j}}}.
$$
Again, by Lemmas \ref{lem:2.3} and \ref{lem:2.4},
 we get the reverse inclusion:
$$
\Kasym{X} = \Kasym{X \otimes F_j^*} \supset \Kasym{\reg{\overline{\mathcal O_j}}}.
$$
By \eqref{eqn:2.5.2}, we obtain Proposition \ref{prop:2.5}.
\end{proof}

We note that $\Kasym{X} =\{0\}$ 
 if and only if $\operatorname{Supp}_K(X)$ is a finite set.  
When $X$ is a \gk-module of finite length, 
 this is equivalent to the condition 
${\mathcal{V}}(X)=\{0\}$, 
 or equivalently,
 $\dim_{\mathbb{C}} X< \infty$.  

\subsection{Transversality of the $K$-supports of two $K$-modules}
\label{subsec:2.6}
In this section
 we formulate the 
\lq\lq{stability of the transversality}\rq\rq\
 of the $K$-supports of two $K$-modules
 under taking the tensor product 
 with finite-dimensional representations.  
For given set $S$, 
 we denote by $\sharp S$ the cardinality of $S$.  

\begin{lemma}
\label{lem:2.6}
Let $X$ and $Y$ be $K$-modules.
\itm{\rm{(1)}}
For any finite-dimensional $K$-module $F$, we have
\[
   \sharp\left(\mSSK{X} \cap \mSSK{Y \otimes F}\right)
   \le
   \dim_{\mathbb{C}}  F \ \sharp\left(\mSSK{X \otimes F^*} \cap \mSSK{Y}\right).
\]
\itm{\rm{(2)}}
The following two conditions are equivalent:
\itm{\hphantom{MM} (i)}
$\sharp\left(\mSSK{X \otimes F^*} \cap \mSSK{Y}\right) < \infty$
 for any finite-dimensional representation $F$ of $K$.
\itm{\hphantom{MM} (ii)}
$\sharp\left(\mSSK{X \otimes F_1} \cap \mSSK{Y \otimes F_2}\right) < \infty$
 for any finite-dimensional representations $F_1$ and $F_2$ of $K$.
\end{lemma}
\begin{proof}
(1)\enspace
Suppose $\mu \in \SSK{X} \cap \SSK{Y \otimes F}$.
Since $V_\mu$ occurs in $V_\nu \otimes F$ for some
 $\nu \in  \SSK{Y}$,
 one finds a weight $v$ of $F$ such that 
\begin{equation}
\label{eqn:2.6.1}
  \mu = \nu + v.
\end{equation}
Then we have
$
 \Hom_K(V_\nu, X \otimes F^*) = \Hom_K(V_\nu \otimes F, X)
 \supset \Hom_K(V_\mu, V_\mu) \neq \{0\}.
$
Hence $\nu \in \SSK{X \otimes F^*}$.
The above consideration yields to a (non-canonical) map
\begin{equation}
\label{eqn:XFY}
 \mSSK{X} \cap \mSSK{Y \otimes F}
 \to
\mSSK{X \otimes F^*} \cap \mSSK{Y},
\quad
 \mu \mapsto \nu
\end{equation}
 with constraints \eqref{eqn:2.6.1}.
The cardinality of
 each fiber of the map \eqref{eqn:XFY} bounded by $\dim F$.
Hence (1) is proved.
\itm{(2)} The second assertion is a direct consequence of (1)
 by setting $F=F_1 \otimes F_2^{\ast}$.  
\end{proof}

\subsection{Admissible restriction and regular functions
 on $K_{\mathbb{C}}/K_{\mathbb{C}}'$}
\label{subsec:2.7}
Let $K'$ be a closed subgroup of a compact Lie group $K$,
 and $K'_\Bbb C \subset K_\Bbb C$ be their complexifications.  
In this section
 we relate $K'$-admissibility of the restriction 
 of a $K$-module with the $K$-support of the space
 ${\mathbb{C}}[K_{\mathbb{C}}/K_{\mathbb{C}}']$
 of regular functions on $K_{\mathbb{C}}/K_{\mathbb{C}}'$.  

\begin{lemma}
\label{lem:2.7}
The following three conditions on a $K$-module $X$ are
 equivalent:
\itm{\,\,(i)}
 $X$ is \adm{K'}.
\itm{\,(ii)}
$X \otimes F'$ is \adm{K'}
 for any finite-dimensional representation  $F'$ of $K'$.
\itm{(iii)}
 $X$ is \adm{K}, 
 and for any finite-dimensional representation  $F$ of $K$,
\begin{equation}
\label{eqn:SStrans}
 \sharp \left(\mSSK{X \otimes F} 
 \cap \mSSK{\Bbb C[K_{\Bbb C}/K'_{\Bbb C}]}\right) < \infty.
\end{equation}
\end{lemma}
\begin{proof}
The implication (i) $\Leftarrow$ (ii) is obvious.
\newline
(i) $\Rightarrow$ (ii):
Suppose (i) holds.
Then for any $\tau \in \widehat{K'}$, we have
$$
   \dim_{\mathbb{C}} \Hom_{K'}(\tau, X \otimes F')
  =
   \dim_{\mathbb{C}} \Hom_{K'}(\tau \otimes (F')^*, X) < \infty
$$
 because 
 $\tau \otimes (F')^*$ is a finite direct sum of irreducible $K'$-modules.
Hence (ii) is proved.
\newline
(ii) $\Rightarrow$ (iii):\enspace  
The $K$-admissibility is obvious from the $K'$-admissibility.  
Let us verify \eqref{eqn:SStrans}.  
Let ${\bf{1}}$ denote the one-dimensional trivial representation of $K$.  
Then we have
\[
   \sharp \{ \mu \in \operatorname{Supp}_K (X \otimes F):
            \Hom_{K'}({\bf{1}}, \mu|_{K'}) \ne \{0\}\}
  \le 
   \dim_{\mathbb{C}} \Hom_{K'}({\bf{1}}, X \otimes F), 
\]
which is finite by the condition (ii).  
Hence \eqref{eqn:SStrans} holds. 
\newline
(iii) $\Rightarrow$ (ii):
Fix any $\tau \in \widehat{K'}$, and any finite-dimensional representation
 $F$ of $K$.
Let $\Ind_{K'}^K \tau$ be an (algebraically) induced representation.
We define a subset of $\widehat{K}$ by
\begin{equation}
\label{eqn:2.7.1}
  {\mathcal{P}}:=   \SSK{\Ind_{K'}^K \tau} \cap \SSK{X \otimes F}.
\end{equation}
We claim ${\mathcal{P}}$ is a finite set.
To see this, we take a finite-dimensional $K$-module $F_1$
 such that $\Hom_{K'}(\tau, F_1|_{K'}) \neq \{0\}$.
Then, we have 
$$
     \Ind_{K'}^K \tau \le_K 
     \Ind_{K'}^K (F_1|_{K'})
     \simeq
     \Bbb C[K_{\Bbb C}/K'_{\Bbb C}] \otimes F_1
$$ 
as $K$-modules.
In particular,
 we have
\begin{equation}
\label{eqn:2.7.2}
   {\mathcal{P}} \subset 
   \SSK{\Bbb C[K_\mathbb C/K'_\mathbb C] \otimes F_1}
    \cap  \SSK{X \otimes F}.  
\end{equation}
The right-hand side of \eqref{eqn:2.7.2} is a finite set
 by the assumption (iii) and Lemma \ref{lem:2.6} (2).
Therefore, ${\mathcal{P}}$ is a finite set.

Next, let us consider the following equation:
\begin{equation}
\label{eqn:2.7.3}
  \dim_{\mathbb{C}} \Hom_{K'}(\tau, X \otimes F) 
= \sum_{\mu \in \widehat{K}}
     \dim_{\mathbb{C}} \Hom_{K'}(\tau, \mu)
    \ \dim_{\mathbb{C}} \Hom_{K}(\mu, X \otimes F).  
\end{equation}
The summation in \eqref{eqn:2.7.3} is actually taken 
 over the finite set ${\mathcal{P}}$.
Furthermore,
 each summand is finite because $X \otimes F$ is \adm{K}.
Hence, \eqref{eqn:2.7.3} is finite.
This means that $X \otimes F$ is \adm{K'}.
Since $F$ is an arbitrary finite-dimensional representation  of $K$,
 (ii) follows.
\end{proof}

\subsection{Proof of Theorem \ref{thm:A}}
\label{subsec:2.9}
We are ready to complete the proof of the main result of this article.  

\begin{proof}
[Proof of Theorem \ref{thm:A}]
Let ${\mathcal{V}}(X)$ be the associated variety of a \gk-module $X$, 
 and ${\mathcal{V}}(X)=\overline{\mathcal{O}_1} \cup \cdots \cup \overline{\mathcal{O}_N}$
 the decomposition into irreducible components
 ({\it{cf}}.~\cite{xkosrallis}).  
By Lemma \ref{lem:2.1}, 
 there are finite subsets
 $S_1$, $\cdots$, $S_N$ and $T$ such that
\[
\begin{cases}
   &\operatorname{Supp}_K({\mathbb{C}}[\overline{\mathcal{O}_j}])
   =  
  \posspan{Z} S_j\quad (1 \le j \le N), 
\\
  &\operatorname{Supp}_K({\mathbb{C}}[K_{\mathbb{C}}/K_{\mathbb{C}}'])
   =  
  \posspan{Z} T.  
\end{cases}
\]
 In place of the conditions (i) and (ii) in Theorem \ref{thm:A},
 we consider the following conditions:
\newline
(i)$'$: $\sharp\left( \SSK{X \otimes F} \cap \SSK{\reg{K_\mathbb C/K'_{\mathbb C}}}
 \right) < \infty$ for any finite-dimensional representation $F$ of $K$.
\newline
(ii)$'$: $\posspan{R}S_j \cap \posspan{R}T = \{0\}$ for any $j = 1, \dots, N$.

We already know the equivalence (i) $\Leftrightarrow$ (i)$'$ from Lemma \ref{lem:2.7},
 and the equivalence (ii) $\Leftrightarrow$ (ii)$''$ from
 Propositions \ref{prop:2.2} and \ref{prop:2.5}.  
Thus, the proof of Theorem~\ref{thm:A} will be completed if we show
 the equivalence (i)$'$  $\Leftrightarrow$ (ii)$'$.
\newline
(i)$'$ $\Rightarrow$ (ii)$'$:
If (i)$'$ holds, then Lemma \ref{lem:2.3} implies 
$$
  \sharp
 \left(\SSK{\reg{\overline{\mathcal O_j}}} \cap \SSK{\reg{K_\mathbb C/K'_\mathbb C}}\right)
 < \infty, 
$$
or equivalently,
$
 \sharp\left(\posspan{Z}S_j \cap \posspan{Z}T\right) < \infty, 
$
 whence the condition (ii)$'$ follows from Lemma \ref{lem:2.8}.
\newline
 (ii)$'$  $\Rightarrow$ (i)$'$:\enspace
Let $F_j$ be as in Lemma \ref{lem:2.3}.
It follows from \eqref{eqn:2.3.2} that
\begin{align*}
     \SSK{X \otimes F}
   &\subset
     \bigcup_{j=1}^N \SSK{\reg{\overline{\mathcal O_j}}\otimes F_j \otimes F}. 
\intertext{Take 
 $\delta := \max\set{\|\nu\|}{\nu \text{ is a weight of } F_j \otimes F
\text{ for some } j}$. 
Then,}
   \bigcup_{j=1}^N \SSK{\reg{\overline{\mathcal O_j}}\otimes F_j \otimes F}&\subset
     \bigcup_{j=1}^N \delta \text{-neighborhood of }
        \SSK{\reg{\overline{\mathcal O_j}}}   
\\
   &\subset
     \bigcup_{j=1}^N
    \delta \text{-neighborhood of }
  \posspan{R}S_j.
\end{align*}
Since the condition (ii)$'$ implies 
that  the intersection of $\posspan{R} T$ with
 any $\delta$-neighborhood of  $\posspan{R}S_j$
 is relatively compact (Lemma \ref{lem:2.8}), 
we get 
$$
 \sharp \left( \SSK{X\otimes F} \cap \posspan{Z} T\right) < \infty.
$$
This shows the implication (ii)$' \Rightarrow$ (i)$'$.
Hence Theorem~\ref{thm:A} is proved.
\end{proof}

\subsection{Proof of Corollary \ref{cor:HCdeco}}
\label{subsec:cor}
\begin{proof}
[Proof of Corollary \ref{cor:HCdeco}]
The implication (i)$'$ $\Rightarrow$ (i) is obvious, 
 and the reverse implication (i) $\Rightarrow$ (i)$'$ follows from 
 the fact that any discrete summand in the restriction 
 $\pi|_{G'}$ for $\pi \in \operatorname{Disc}(G)$
 belongs to $\operatorname{Disc}(G')$, 
 see \cite[Cor.~8.7]{xkdisc}.  
Then the implication (i)$'$ $\Rightarrow$ (ii) follows from the fact 
 that for every $\mu \in \widehat{K'}$ 
 there are at most finitely many elements 
 in $\operatorname{Disc}(G')$
 having $\mu$ as a $K'$-type,
 whereas the implication (ii) $\Rightarrow$ (i) 
 is proved in \cite[Thm.~1.2]{xkdecomp}.  
Since the equivalence (ii) $\Leftrightarrow$ (iii) 
 holds by Theorem \ref{thm:A}, 
 Corollary \ref{cor:HCdeco} is proved.  
\end{proof}

\section{\gk-modules with finite weight multiplicities}
\label{sec:convex}

In this section,
 we relate weight multiplicities
 for \gk-modules
 with celebrated Kostant's convexity theorem
 \cite{xkostant}.  

\subsection{Simple Lie groups of (non)Hermitian type}
\label{subsec:GKZK}
Let $G$ be a real reductive linear Lie group,
 $K$ a maximal compact subgroup, 
 $Z_K$ the center of $K$, 
 and $T^s$ a maximal torus of the derived group 
 $K^s:=[K, K]$.  
Then $T:=T^s Z_K$ is a maximal torus of $K$.  
When $G$ is a simple Lie group,
 $Z_K$ is at most one-dimensional.  

A simple Lie group $G$
 (or its Lie algebra ${\mathfrak{g}}$)
 is called
 {\it{of Hermitian type}}, 
 if $Z_K$ is one-dimensional, 
 or equivalently, 
 if the associated Riemannian symmetric space $G/K$
 is a Hermitian symmetric space.  
It is the case
 when the Lie algebra ${\mathfrak{g}}$ is 
 ${\mathfrak{s u}}(p,q)$, ${\mathfrak{s o}}(2 n)$, 
 ${\mathfrak{s o}}^{\ast}(2n)$, ${\mathfrak{s p}}(n,{\mathbb{R}})$, 
 ${\mathfrak{e}}_{6(-14)}$, or ${\mathfrak{e}}_{7(-25)}$,  
 whereas ${\mathfrak{g}}$ $=$ ${\mathfrak{s l}}(n, {\mathbb{R}})$ $(n\ne 2)$, 
 ${\mathfrak{s o}}(p,q)$ $(p,q \ne 2)$, 
 ${\mathfrak{s u}}^{\ast}(2n)$, 
 ${\mathfrak{s p}}(p,q)$, ${\mathfrak{s l}}(n,{\mathbb{C}})$,  
 ${\mathfrak{s o}}(n,{\mathbb{C}})$, or ${\mathfrak{s p}}(n,{\mathbb{C}})$
 are not of Hermitian type.  

\subsection{Admissibility for the restriction to toral subgroups}
\label{subsec:multT}
In contrast to ${\mathfrak{g}}$-modules
 in the BGG category ${\mathcal{O}}$, 
 there are not many \gk-modules
 with finite weight multiplicities.  
We formulate this feature as follows.  

\begin{theorem}
\label{thm:Tmult}
Suppose that $X$ is \gk-module 
 of finite length.  
If $\dim_{\mathbb{C}} X=\infty$ then 
$
\dim_{\mathbb{C}} \operatorname{Hom}_{T^s}(\chi, X)=\infty
$
 for some $\chi \in \widehat{T^s}$. 
\end{theorem}

We shall see that Theorem \ref{thm:Tmult} 
 is derived from Kostant's convexity theorem 
 (Fact \ref{fact:Kos})
 and from Theorem \ref{thm:A}.  
The following two corollaries 
 for simple Lie groups $G$ are immediate consequence
 of Theorem \ref{thm:Tmult}
 and its proof
(Section \ref{subsec:Kos}).  

\begin{corollary}
\label{cor:nonHerm}
Suppose that $G$ is not of Hermitian type.  
Then for any infinite-dimensional irreducible \gk-module $X$, 
 there exists $\chi \in \widehat T$ such that 
$\dim_{\mathbb{C}} \operatorname{Hom}_{T}(\chi, X)=\infty$.  
\end{corollary}

\begin{corollary}
\label{cor:HermT}
Suppose that $G$ is of Hermitian type,
 and $X$ a \gk-module of finite length.  
Then $X$ is $T$-admissible
 if and only if $X$ is $Z_K$-admissible. 
\end{corollary}

\begin{remark}
{\rm{
An irreducible \gk-module $X$ is called a highest weight module
 if $X$ is ${\mathfrak{b}}$-finite
 for some Borel subalgebra ${\mathfrak{b}}$
 of ${\mathfrak{g}}_{\mathbb{C}}={\mathfrak{g}} \otimes_{\mathbb{R}} {\mathbb{C}}$. 
There exist infinite-dimensional irreducible highest weight \gk-modules
 if and only if $G$ is of Hermitian type.  
In this case any such $X$ is $Z_K$-admissible
 (see \cite[Rem.~3.5 (3)]{xkdecoalg}), 
 hence $X$ is also $T$-admissible.  
}}
\end{remark}

Corollary \ref{cor:HermT} fits well into the Kirillov--Kostant--Duflo orbit
 philosophy
 (see \cite{xdv, xkirillov, xknasrin, xkostant70, para2015, para2019} for instance):

\begin{proposition}
\label{prop:KNproper}
Suppose $G$ is a simple Lie group of Hermitian type,
 and ${\mathcal{O}}$ a coadjoint orbit in ${\mathfrak{g}}^{\ast}$.  
Then the following two conditions are equivalent:
\itm{\rm{(i)}}
The momentum map ${\mathcal{O}} \to {\mathfrak{t}}^{\ast}$ is proper.  
\itm{\rm{(ii)}}
The momentum map ${\mathcal{O}} \to {\mathfrak{z}}_{\mathfrak{k}}^{\ast}$
 is proper.  
\end{proposition}

\subsection{An application of Kostant's convexity theorem}
\label{subsec:Kos}
Suppose $K$ is a connected compact Lie group,
 and $T$ is a maximal torus of $K$.  
Let $W_K$ be the Weyl group 
 for the root system $\Delta({\mathfrak{k}}_{\mathbb{C}}, {\mathfrak{t}}_{\mathbb{C}})$.  
By a $K$-invariant inner product $\langle\, , \rangle$ on ${\mathfrak{k}}$, 
 we identify ${\mathfrak{t}}^{\perp}$ $(\subset {\mathfrak{k}}^{\ast})$ 
 with the orthogonal complementary subspace of ${\mathfrak{k}}$, 
 and write
$
  \pr {\mathfrak{k}}{\mathfrak{t}} \colon {\mathfrak{k}} \to {\mathfrak{t}}
$
 for the projection 
 with respect to the direct sum decomposition 
 ${\mathfrak{k}}={\mathfrak{t}}\oplus {\mathfrak{t}}^{\perp}$.

For a finite subset $S=\{s_1, \cdots, s_k\}$ of ${\mathfrak{t}}$, 
 the convex hull of $S$ is the smallest convex set containing $S$, 
 which is expressed as:
\[
\operatorname{Conv}(S)
:=
\left\{\sum_{i=1}^k a_i s_i: 
a_1, \cdots, a_k \ge 0,
\quad
 a_1+ \cdots +a_k =1
\right\}.  
\]

We recall Kostant's convexity theorem:
\begin{fact}
[{\cite[Thm.~8.2]{xkostant}}]
\label{fact:Kos}
For any $Y \in {\mathfrak{t}}$, 
 we have 
$
\pr{\mathfrak{k}}{\mathfrak{t}}(\Ad(K)Y)
=
\operatorname{Conv}(W_K Y).  
$
\end{fact}

Fact \ref{fact:Kos} determines the momentum set $\Delta(T^{\ast}(K/T))$
 of the cotangent bundle of the flag manifold $K/T$
 as follows:

\begin{proposition}
\label{prop:191357}
Suppose that $K$ is a connected semisimple compact Lie group.  
Then 
\[
   \Delta(T^{\ast}(K/T))=C_K(T)={\mathfrak{t}}_+^{\ast}.  
\]
\end{proposition}

\begin{proof}
Fix a nonzero element $Y \in {\mathfrak{t}}$.  
Then Kostant's convexity theorem shows
 that $\pr {\mathfrak{k}}{\mathfrak{t}} (\Ad(K)Y)$ contains 
 the origin $0$.  
In particular, 
 there exists $k \in K$
 such that 
$
  Y':= \Ad(k) Y \in {\mathfrak{t}}^{\perp}
$.  
This means that 
$
  Y \in \Ad(K){\mathfrak{t}}^{\perp}
$, 
hence 
$
 \pr {\mathfrak{k}}{\mathfrak{t}} 
 (\Ad(K) {\mathfrak{t}}^{\perp})= {\mathfrak{t}}.  
$
By \eqref{eqn:2.2.3}, 
 we get Proposition \ref{prop:191357}.  
\end{proof}

\begin{proof}
[Proof of Theorem \ref{thm:Tmult}]
Applying Proposition \ref{prop:191357} to $K^s/T^s$, 
 we obtain 
\[
C_K(T^s)={\mathfrak{t}}_+^{\ast}
\]
 because $K=K^s Z_K$.  
In turn, 
 Theorem \ref{thm:A} tells that $X$ is $T^s$-admissible
 if and only if $\Kasym{X}=\{0\}$, 
 or equivalently,
 $\dim X< \infty$.   
\end{proof}

\begin{proof}
[Proof of Corollary \ref{cor:nonHerm}]
Immediate from Theorem \ref{thm:Tmult}
 because $T=T^s$.  
\end{proof}

\begin{proof}
[Proof of Corollary \ref{cor:HermT}]
We regard $({\mathfrak{t}}^s)^{\ast}$
 as a subspace of ${\mathfrak{t}}^{\ast}$
 via the direct sum decomposition 
 ${\mathfrak{t}}= {\mathfrak{t}}^s \oplus {\mathfrak{z}}_{\mathfrak{k}}$.  
By Proposition \ref{prop:191357}, 
 we have
$C_K(T)={\mathfrak{t}}_+^{\ast} \cap ({\mathfrak{t}}^s)^{\ast}
 =C_K(Z_K)$, 
 whence Corollary \ref{cor:HermT}.  
\end{proof}

\section{Admissible restriction of degenerate principal series representations}
\label{sec:4}
In the orbit philosophy due to Kirillov--Kostant, 
 the Zuckerman derived functor modules $A_{\mathfrak{q}}(\lambda)$
 are supposed to be attached 
 to elliptic coadjoint orbits,
 whereas parabolically induced representations
 $\operatorname{Ind}_Q^G({\mathbb{C}}_{\lambda})$ 
are to hyperbolic coadjoint orbits.  
Classification theory of admissible restrictions
 has been developed mainly 
 for $A_{\mathfrak{q}}(\lambda)$, 
 see \cite{xdgv, xkdecomp, xkdecoass, deco-euro, decoAq, xkyoshi15}
 for example.  
In this section 
 we apply Theorem \ref{thm:A}
 to induced representations from 
 a parabolic subgroup $Q$ of $G$
 and to their subquotient modules ({\it{$Q$-series}}).

\subsection{Irreducible representations in the $Q$-series}
\label{subsec:defQ}
Suppose that $Q$ is a parabolic subgroup 
 of a reductive Lie group $G$.  
\begin{definition}
\label{def:Qseries}
An irreducible admissible representation $\pi$ of $G$
 is said to be in the {\it{$Q$-series}}
 if $\pi$ occurs as a subquotient 
 of the induced representation 
 $\operatorname{Ind}_Q^G \tau$ from a finite-dimensional representation $\tau$ of $Q$.  
\end{definition}

\begin{example}
When $Q=G$, 
 $\pi$ is in the $Q$-series
 if and only if $\dim_{\mathbb{C}} \pi < \infty$.  
\end{example}

\begin{example}
\label{ex:QGP}
When $Q$ is a minimal parabolic subgroup $P$, 
 any irreducible admissible representation  of $G$
 belongs to the $Q$-series 
 by Harish-Chandra's subquotient theorem.  
\end{example}

The next example is a generalization of Example \ref{ex:QGP}.  

\begin{example}
Let $G/H$ be a reductive symmetric space,
 that is, 
 $H$ is an open subgroup of $G^{\sigma}=\{g \in G: \sigma g=g\}$
 for some involutive automorphism $\sigma$
 of a real reductive Lie group $G$.  
Take a Cartan involution $\theta$ of $G$
 commuting with $\sigma$,  
 and a maximal abelian subspace ${\mathfrak {a}}$
 in ${\mathfrak{g}}^{-\sigma,-\theta}=
\{X \in {\mathfrak{g}}:\sigma X=\theta X=-X\}$.  
Let $Q$ be a parabolic subgroup of $G$
 defined by a generic element $X \in {\mathfrak{a}}$, 
 that is, 
 $Q$ is the normalizer of the real parabolic subalgebra:
\[
{\mathfrak{q}}=\text{the sum of the eigenspaces of $\operatorname{ad}(X)$
 with nonnegative eigenvalues}.  
\]
Such $Q$ is uniquely determined up to conjugation
 by an element of $G$.  
We say that $Q$ is a {\it{minimal parabolic subgroup for $G/H$}}.  
\end{example}

\begin{remark}
{\rm{
Let $G/H$ be a reductive symmetric space, 
 and $Q$ a minimal parabolic subgroup for $G/H$.  
Then any irreducible representation 
 that can be realized as a subquotient
 in the regular representation on $C^{\infty}(G/H)$
 belongs to the $Q$-series.  
}}
\end{remark}

\subsection{Restriction of representations in the $Q$-series}
\label{subsec:Qseries}
We give a necessary and sufficient condition 
 for all irreducible representations
 in the $Q$-series 
 to be $K'$-admissible 
 where $K'$ is a (not necessarily, maximal) compact subgroup.  

\begin{theorem}
\label{thm:0.2}
Let $G$ be a real reductive linear Lie group, 
 $K$ a maximal compact subgroup, 
 $K'$ a closed subgroup of $K$, 
 and $Q$ a parabolic subgroup of $G$.  
Then the following two conditions are equivalent:
\itm{\rm{(i)}}
for any irreducible representation $\pi$ of $G$ in the $Q$-series, 
 $\pi|_{K'}$ is $K'$-admissible;
\itm{\rm{(ii)}}
$C_K(Q \cap K) \cap C_K(K') = \{0\}$.  
\end{theorem}

\begin{proof}
Since the induced representation $\operatorname{Ind}_Q^G(\tau)$
 is of finite length as a $G$-module, 
 the condition (i) is equivalent to the following condition:
\newline
(i)$'$\enspace 
 $\operatorname{Ind}_Q^G(\tau)$ is $K'$-admissible
 for any finite-dimensional representation $\tau$ of $Q$.  
\newline
By Proposition \ref{prop:2.2} and Lemma \ref{lem:2.4}, 
 the asymptotic $K$-support of $\operatorname{Ind}_Q^G(\tau)$ is given by
\begin{equation}
\label{eqn:ASQ}
 \operatorname{A S}_K(\operatorname{Ind}_Q^G(\tau))
 =
  \operatorname{A S}_K(\operatorname{Ind}_{Q \cap K}^K(\tau|_{Q \cap K}))
 =
  \operatorname{A S}_K(\operatorname{Ind}_{Q \cap K}^K({\bf{1}}))
 =C_K(Q \cap K).  
\end{equation}
Hence Theorem \ref{thm:0.2} is derived from Theorem \ref{thm:A}. 
\end{proof}

Let $P=M A N$ be a minimal parabolic subgroup of $G$.  
Applying Theorem \ref{thm:0.2} to the case $Q=P$, 
 we obtain from Example \ref{ex:QGP} the following:
\begin{corollary}
\label{cor:4.6}
Let $K'$ be a closed subgroup of $K$.  
Then the following two conditions are equivalent:
\itm{\rm{(i)}}
any irreducible admissible representation of $G$ 
 is $K'$-admissible;
\itm{\rm{(ii)}}
$C_K(M) \cap C_K(K')=\{0\}$.  
\end{corollary}

\begin{remark}
When $G$ is of real rank one, 
 then $K/M$ is isomorphic to a sphere.  
In this case, Vargas \cite{xvargas10} classified all subgroups $K'$
 satisfying the condition in Corollary \ref{cor:4.6}.  
\end{remark}

\begin{example}
\label{ex:SOUpq}
Let $G=SO(2p,2q)$, 
 and $K' = U(p) \times U(q)$.  
Suppose $Q$ is a parabolic subgroup of $G$
 with Levi subgroup 
 $L \simeq S O(2p-1, 2q-1) \times G L (1,{\mathbb{R}})$.  
Then $Q \cap K=L \cap K$, 
 and via the standard basis
 of ${\mathfrak{t}}^{\ast} \simeq {\mathbb{R}}^{p+q}$, 
\begin{align*}
C_K (Q \cap K)=&
\{(a,0,\cdots,0;b,0,\cdots,0): a, b \ge 0\}, 
\\
C_K (K')=&
\{(x_1,x_1,\cdots, x_{[\frac p 2]}, x_{[\frac p 2]}, (0);y_1,y_1,\cdots, y_{[\frac q 2]}, y_{[\frac q 2]}, (0)):
\\
&
\hphantom{MMMMMMMMMMMMMMM}
 x_1 \ge x_2 \ge \cdots,y_1 \ge y_2 \ge \cdots \}.   
\end{align*}
Hence $C_K(Q \cap K) \cap C_K(K') =\{0\}$. 
Thus the criterion (ii) in Theorem \ref{thm:0.2} is fulfilled.  
Let $G'=U(p,q)$ be the natural subgroup of $G$
 containing $K'$.  
Then for any irreducible unitary representation $\pi$ of $G$
 in the $Q$-series
 is $G'$-admissible 
 when restricted to the subgroup $G'$
 because it is $K'$-admissible.  
See \cite{xhowetan} and \cite{xkdecomp} 
 for branching laws of representations $\pi$ in the $Q$-series
 with respect to the pair $(G,G')=(SO(2p,2q), U(p,q))$.  
\end{example}

In Example \ref{ex:SOUpq}, 
 the two polyhedral cones
 $C_K(Q \cap K)$ and $C_K(K')$ are easy 
to compute, 
 in particular, 
 because both $(K, Q \cap K)$ and $(K,K')$ are symmetric pairs.  
In the next section, 
 we recall some useful general facts for this.

\subsection{Momentum set $\Delta(T^{\ast}(K/K'))$
for symmetric pair}
\label{subsec:Csymm}
Suppose that $\sigma$ is an involutive automorphism
 of a connected compact Lie group $K$.  
We use the same letter $\sigma$ to denote its differential, 
 and write ${\mathfrak{k}}={\mathfrak{k}}^{\sigma}+{\mathfrak{k}}^{-\sigma}$
 for the eigenspace decomposition of $\sigma$
 with eigenvalues $+1$ and $-1$.  
We take a $\sigma$-stable Cartan subalgebra ${\mathfrak{j}}$
 of ${\mathfrak{k}}$ 
such that ${\mathfrak{j}}^{-\sigma}$ is a maximal abelian subspace 
 of ${\mathfrak{k}}^{-\sigma}$, 
 and fix a positive system 
 $\Sigma^+({\mathfrak{k}}_{\mathbb{C}}, {\mathfrak{j}}_{\mathbb{C}}^{-\sigma})$
 of the restricted root system 
 $\Sigma({\mathfrak{k}}_{\mathbb{C}}, {\mathfrak{j}}_{\mathbb{C}}^{-\sigma})$.  
Choose a positive system 
 $\Delta^+({\mathfrak{k}}_{\mathbb{C}}, {\mathfrak{j}}_{\mathbb{C}})$
 compatible with $\Sigma^+({\mathfrak{k}}_{\mathbb{C}}, {\mathfrak{j}}_{\mathbb{C}}^{-\sigma})$
 in the following sense:
\[
  \{\alpha|_{{\mathfrak{j}}_{\mathbb{C}}^{-\sigma}}
   :
   \alpha \in \Delta^+({\mathfrak{k}}_{\mathbb{C}}, {\mathfrak{j}}_{\mathbb{C}})\}
  \setminus
  \{0\}
  =
{\Sigma}^+({\mathfrak{k}}_{\mathbb{C}}, {\mathfrak{j}}_{\mathbb{C}}^{-\sigma}).
\]
Let $({\mathfrak{j}}^{-\sigma})_+^{\ast}$ and 
 ${\mathfrak{j}}_+^{\ast}$ be the dominant chamber
 for $\Sigma^+({\mathfrak{k}}_{\mathbb{C}}, {\mathfrak{j}}_{\mathbb{C}}^{-\sigma})$
 and $\Delta^+({\mathfrak{k}}_{\mathbb{C}}, {\mathfrak{j}}_{\mathbb{C}})$, 
 respectively.  
We may regard $({\mathfrak{j}}^{-\sigma})_+^{\ast} \subset {\mathfrak{j}}_+^{\ast}$
 according to the direct decomposition
 ${\mathfrak{j}} ={\mathfrak{j}}^{\sigma} \oplus {\mathfrak{j}}^{-\sigma}$.  
When a positive system $\Delta^+({\mathfrak{k}}_{\mathbb{C}}, {\mathfrak{t}}_{\mathbb{C}})$
 is given independently of $\sigma$, 
 we choose an inner automorphism of ${\mathfrak{k}}$  
 which induces bijections $\iota \colon {\mathfrak{t}} \overset \sim \to {\mathfrak{j}}$
 and $\iota^{\ast} \colon \Delta^+({\mathfrak{k}}_{\mathbb{C}}, {\mathfrak{j}}_{\mathbb{C}}) \overset \sim \to \Delta^+({\mathfrak{k}}_{\mathbb{C}}, {\mathfrak{t}}_{\mathbb{C}})$, 
 and set
\[
   ({\mathfrak{t}}^{-\sigma})_+^{\ast} := \iota^{\ast}(({\mathfrak{j}}^{-\sigma})_+^{\ast}) \,\subset \sqrt{-1} {\mathfrak{t}}^{\ast}.  
\]

\begin{proposition}
\label{prop:Csymm}
Suppose $(K,K')$ is a symmetric pair
 defined by an involutive automorphism $\sigma$.  
Then 
$
  \Delta(T^{\ast}(K/K'))=C_K(K')=({\mathfrak{t}}^{-\sigma})_+^{\ast}.  
$
\end{proposition}

\begin{remark}
Suppose $K$ is a maximal compact subgroup 
 of a connected real reductive Lie group, 
 and $Q$ a standard parabolic subgroup.  
If the unipotent radical of $Q$ is abelian, 
 then $(K,Q \cap K)$ forms a symmetric pair, 
 and therefore we can apply also Proposition \ref{prop:Csymm}
 to the computation 
 of $C_K(Q \cap K)$ in Theorem \ref{thm:0.2}.  
\end{remark}

\subsection{Boundaries of spherical varieties
 with hidden symmetries}
\label{subsec:overgp}

As typical examples of Theorem \ref{thm:0.2}, 
 we formulate the following theorem
 motivated by analysis 
 on standard pseudo-Riemannian locally symmetric spaces 
 $\Gamma \backslash G/H$
 (\cite{xpoincare, Specstandard}):
\begin{theorem}
\label{thm:overgp}
Let $G/H$ be a symmetric space
 with $G$ simple Lie group, 
 and $Q$ a minimal parabolic subgroup for the symmetric space $G/H$.  
Let $G'$ be a reductive subgroup of $G$
 acting properly on $G/H$, 
 such that $G_{\mathbb{C}}/H_{\mathbb{C}}$ is 
 $G_{\mathbb{C}}'$-spherical.  
Then any irreducible admissible representation $\pi$ of $G$
 in the $Q$-series is $K'$-admissible.  
In particular, 
 the restriction $\pi|_{G'}$ is infinitesimally discretely decomposable 
 in the sense of \cite[Def.~4.2.3]{deco-euro}.  
\end{theorem}

This theorem is a counterpart
 of \cite[Thm.~5.1]{xkhowe70}
 where $\pi$ was assumed to be a subquotient 
 of the regular representation of $G$
 in the space ${\mathcal{D}}'(G/H)$
 of distributions on $G/H$.

In the setting of Theorem \ref{thm:overgp}, 
 the symmetric space $G/H$ admits 
 a compact Clifford--Klein form $\Gamma \backslash G/H$
 as the quotient by a torsion-free cocompact subgroup $\Gamma$ in $G'$.  
The classification of the triples $(G,H,G')$
 in Theorem \ref{thm:overgp} is given in \cite{Specstandard}.  
Applications of Theorem \ref{thm:overgp}
 will be discussed in subsequent papers.  
In this article, 
 we illustrate Theorem \ref{thm:overgp}
 only by some examples:
\begin{example}
\label{ex:OOU}
The triple $(G,H,G')=(SO(2p,2q), SO(2p-1,2q), U(p,q))$ satisfies
 the assumptions of Theorem \ref{thm:0.2}.  
In this case, 
 Example \ref{ex:SOUpq} is recovered.  
\end{example}

\begin{example}
\label{ex:D4}
The triple $(G,H,G')=(SO(8,8), SO(7,8), Spin(1,8))$ satisfies
 the assumption of Theorem \ref{thm:0.2}.  
Via the standard basis of ${\mathfrak{t}}^{\ast} \simeq {\mathbb{R}}^8$, 
 we may write as 
\begin{align*}
C_K(Q \cap K) &= \{(a,0,0,0;b,0,0,0): a, b \ge 0\}, 
\\
C_K(K') &= \{((x_1,x_2,x_3,x_4), \zeta(x_1,x_2,x_3,-x_4)):x_1 \ge x_2 \ge x_3 \ge |x_4| \}, 
\end{align*}
where $\zeta$ is an outer automorphism of order 3
 for the root system $D_4$.  
Thus the criterion (ii) in Theorem \ref{thm:0.2} is fulfilled, 
 and Theorem \ref{thm:overgp} is verified in this case.  
Explicit branching laws of irreducible square-integrable representations
 in the $Q$-series
 with respect to $(G,G')=(SO(8,8), Spin(1,8))$
 are obtained in \cite[Thm.~5.5]{xkhowe70}
 and in \cite{xSTV}.  
\end{example}

\vskip 1pc
\centerline{{\bf{Acknowledgement}}}
Theorem \ref{thm:A} was motivated
 when I was visiting Harvard University
 during the academic year 2000-2001.
I thank all the audience of my graduate course there,
 who patiently attended the class and raised 
 a question if the reverse implication (ii) $\Rightarrow$ (i)
 in Theorem \ref{thm:A} holds
 when I was explaining a proof for (i) $\Rightarrow$ (ii)
 by a different approach based on micro-local analysis.  
Theorem \ref{thm:A} was then proved 
 by a change of machinery, 
 and announced in the Proceedings of ICM 2002 \cite[Thm.~D]{xkICM2002}. 
A sketch of the proof 
 was given in the lecture notes
 \cite[Chap.~6]{deco-euro}
 which was prepared during the special year program 
 \lq\lq{Representation Theory 
 of Lie groups 2002}\rq\rq\ 
 at IMS-NUS organized by E.-T.~Tan and C.-B.~Zhu.  
I would like to extend my thanks to M.~Duflo, 
B.~Kostant, T.~Kubo, 
B.~{\O}rsted, 
W.~Schmid, 
M.~Vergne, 
 and D.~Vogan
 for their comments on this work in various occasions. 
I also would like to mention that since then, 
 there has been also interesting progress from the geometric aspect
 of Hamiltonian actions 
 and their quantum analogues, 
 see Paradan \cite{para2015, para2019}
 in the case of discrete series representations, 
 for instance.

This work was partially supported
 by Grant-in-Aid for Scientific Research (A) (JP18H03669), 
Japan Society for the Promotion of Science.

\end{document}